\documentclass[12pt, reqno]{amsart}
\usepackage{amssymb,amsthm,amsmath}
\usepackage[numbers,sort&compress]{natbib}
\usepackage{amssymb,amsmath}
\usepackage{amsfonts}
\usepackage{mathrsfs}
\usepackage{latexsym}
\usepackage{amssymb}
\usepackage{amsthm}
\usepackage{indentfirst}
\date{\today}

\usepackage{color}

\let\oldsection\section
\renewcommand\section{\setcounter{equation}{0}\oldsection}

\newtheorem{theorem}{Theorem}[section]
\newtheorem{lemma}{Lemma}[section]
\newtheorem{proposition}{Proposition}[section]

\newtheorem{remark}{Remark}[section]

\allowdisplaybreaks
\begin{document}

\title[local existence of strong solutions to the compressible N-S eq]{Local well-posedness of isentropic compressible Navier-Stokes equations with vacuum }

\author{Huajun Gong}
\address{College of Mathematics and Statistics, Shenzhen University, Shenzhen, China}
\email{Huajun84@szu.edu.cn}
\author{Jinkai Li}
\address{South China Research Center for Applied Mathematics and Interdisciplinary Studies, South China Normal University, Zhong Shan Avenue West 55, Guangzhou 510631, China}
\email{jklimath@gmail.com}
\author{Xian-gao Liu}
\address{School of Mathematic Sciences, Fudan University, Shanghai, 200433, China}
\email{xgliu@fufan.edu.cn}
\author{Xiaotao Zhang}
\address{South China Research Center for Applied Mathematics and Interdisciplinary Studies, South China Normal University, Zhong Shan Avenue West 55, Guangzhou 510631, China}
\email{{xtzhang@m.scnu.edu.cn}}
\maketitle

\begin{abstract}
In this paper, the local well-posedness of strong solutions to the Cauchy problem of the isentropic compressible Navier-Stokes equations
is proved with the initial date being allowed to have vacuum. The main contribution of this paper is that the well-posedness is
established without assuming any compatibility condition on the initial data, which was widely used before in many literatures
concerning the well-posedness of compressible Navier-Stokes equations in the presence of vacuum.
\end{abstract}

\section{Introduction}

The isentropic compressible Navier-Stokes equations read as
\begin{eqnarray}\label{u}
  \rho (u_t +(u\cdot\nabla)u) -\mu \Delta u -(\lambda +\mu ) \nabla \text{div}\, u +\nabla P=0,\\
\label{rho}  \rho_t +\text{div}\, (\rho u )=0,
\end{eqnarray}
in $\mathbb{R}^3 \times (0,T)$, where the density $\rho\geq0$ and the velocity field $u\in\mathbb R^3$ are the unknowns. Here $P$ is the scalar pressure given as $P=a\rho^{\gamma}$, for two constants $a>0$ and $\gamma >1$. The viscosity constants $\lambda, \mu$ satisfy the physical requirements:
$$\mu>0,\quad 2\mu+3\lambda \geq 0.$$
System (\ref{u})--(\ref{rho}) is complemented with the following initial-boundary conditions
\begin{equation}\label{u-bd1}
\left\{
  \begin{array}{ll}
    (\rho,{\rho} u)|_{t=0} =(\rho_0, {\rho_0} u_0),\\
    u(x,t)\rightarrow 0,  \text{ as } |x|\rightarrow \infty.
  \end{array}
\right.
\end{equation}

There are extensive literatures on the studies of the
compressible Navier-Stokes equations. In the absence of vacuum, that is
the density has positive lower bound, the system is locally
well-posed for large initial data, see, e.g., \cite{NASH62,ITAYA71,VOLHUD72,TANI77,VALLI82,LUKAS84}; however, the
global well-posedness is still unknown. It has been known that system in one dimension is globally well-posed for large initial data, see, e.g., \cite{KANEL68,KAZHIKOV77,KAZHIKOV82,ZLOAMO97,ZLOAMO98,CHEHOFTRI00,JIAZLO04} and the references therein, see \cite{LILIANG16} for the large time behavior of the solutions, and also \cite{LJK18,LIXIN17,LIXIN19} for the global well-posedness for the case that with nonnegative density. For the multi-dimensional case, the global well-posedness holds for small initial data, see, e.g., \cite{MATNIS80,MATNIS81,MATNIS82,MATNIS83,PONCE85,VALZAJ86,DECK92,HOFF97,KOBSHI99,DANCHI01,CHIDAN15,DANXU18,FZZ18}.
In the presence of vacuum, that is the density may vanish
on some set, global existence (but without uniqueness) of weak solutions has been known, see \cite{LIONS93,LIONS98,FEIREISL01,JIAZHA03,FEIREISL04P,FEIREISL04B,BRESCH18}. Local well-posedness of strong solutions was proved for
suitably regular initial data under some extra compatibility conditions (being mentioned in some details below) in \cite{CHOKIM04,CHOKIM06-1,CHOKIM06-2}. In general, when the vacuum is involved, one can only expect solutions in the homogeneous
Sobolev spaces, that is, the $L^2$ integrability of $u$ is not expectable, see \cite{LWX}.
Global well-posedness holds if the initial basic energy is sufficiently small, see \cite{HLX12,LIXIN13,HUANGLI11,WENZHU17}; however, due to the blow-up results in \cite{XIN98,XINYAN13}, the corresponding entropy of the global solutions in \cite{HUANGLI11,WENZHU17} must be
infinite somewhere in the vacuum region, if the initial density is compactly supported.

In this paper, we focus on the well-posedness of the Cauchy problem to system (\ref{u})--(\ref{rho}) in the presence of vacuum.
As mentioned in the previous paragraph, local well-posedness of strong solutions to the compressible Navier-Stokes in the presence
of vacuum has already been studied
in \cite{CHOKIM04,CHOKIM06-1,CHOKIM06-2}, where, among some other conditions, the
regularity assumption
\begin{equation}
  \label{REGULARITY}
\rho_0-\rho_\infty\in H^1\cap W^{1,q},\quad u_0\in D_0^1\cap D^2,
\end{equation}
for some constant $\rho_\infty\in[0,\infty)$, and the compatibility condition
\begin{equation}
  \label{COMPATIBILITY}
  -\mu\Delta u_0-(\mu+\lambda)\nabla\text{div}\,u_0+\nabla P(\rho_0)=\sqrt{\rho_0}g,
\end{equation}
for some $g\in L^2$, were used. Similar assumptions as (\ref{REGULARITY}) and (\ref{COMPATIBILITY}) were also widely used in studying
many other fluid dynamical systems when the vacuum is involved, see, e.g., \cite{HLX12,HUANGLI11,LIXIN13,WENZHU17,CHOEKIM,ZJW,HXDWY,WHW,WHYDSJ,CQTZWYJ,LJK,GHJLJK,GHJLJKXC}.

Assumptions (\ref{REGULARITY}) and (\ref{COMPATIBILITY}) are so widely used when the initial vacuum is taken into
consideration, one may ask if the regularities on the initial data
stated in (\ref{REGULARITY}) can be relaxed and if the compatibility condition (\ref{COMPATIBILITY}) is necessary for the
local well-posedness of strong solutions to the corresponding system.
In a previous work \cite{LJK17}, the second author of this paper considered these questions for
the inhomogeneous incompressible Navier-Stokes equations, and found that the compatibility condition is not necessary for the
local well-posedness. The aim of the current paper is to give the same answer for the isentropic compressible Navier-Stokes equations. As will be shown in this paper that we can indeed
reduce the regularities of the initial velocity in (\ref{REGULARITY}) and remove the
compatibility condition (\ref{COMPATIBILITY}), without loosing the existence
and uniqueness, but the prices that we need to pay are the following:
(i) the corresponding strong solutions do
not have as high regularities as those in \cite{CHOKIM04,CHOKIM06-1,CHOKIM06-2} where both (\ref{REGULARITY}) and (\ref{COMPATIBILITY}) were
assumed; (ii) one can only ask for the continuity, at the initial time, of
the momentum $\rho u$, instead of the velocity $u$ itself.

Before stating our main results, let us introduce some notations. Throughout this paper, we use $L^r=L^r(\mathbb R^3)$ and $W^{k,r}=
W^{k,r}(\mathbb R^3)$ to denote, respectively,
the standard Lebesgue and Sobolev spaces in $\mathbb{R}^3$, where $k$ is a positive integer and $r\in [1,\infty]$. When $r=2$, we use $H^k$
instead of $W^{k,2}$. For simplicity, we use $\|\cdot\|_r=\|\cdot\|_{L^r}$. We denote
$$D^{k,r}=\Big\{ u\in L^1_{loc}(\mathbb{R}^3)\,\Big|\,\|\nabla^k u\|_r<\infty\Big\}, \quad D^k=D^{k,2},$$
$$D^{1}_0=\Big\{ u\in L^6\,\Big|\,\|\nabla u\|_2<\infty \Big\}.$$
For simplicity of notations, we adopt the notation
$$
\int fdx=\int_{\mathbb R^3}fdx.
$$

Our main result is the following:

\begin{theorem}\label{keyth}
Suppose that the initial data $(\rho_0,u_0)$ satisfies
$$\rho_0\geq0,\quad \rho-\rho_\infty\in H^1\cap W^{1,q},\quad u_0\in D_0^1,\quad\rho_0u_0\in L^2,$$
for some $\rho_\infty\in [0, \infty)$ and some $q\in (3,6)$.

Then, there exists a positive time $\mathcal T$, depending only on $\mu$, $\lambda$, $a$, $\gamma$, $q$, and the upper bound of
$\psi_0:=\|\rho_0\|_\infty+\|\rho_0-\rho_\infty\|_2+\|\nabla\rho_0\|_{L^2\cap L^q}+\|\nabla u_0\|_2$, such that system (\ref{u})--(\ref{rho}), subjects to (\ref{u-bd1}), admits a unique solution $(\rho, u)$ on $\mathbb{R}^3\times (0,\mathcal T),$ satisfying
\begin{eqnarray*}
  &\rho-\rho_\infty\in C([0,\mathcal T]; L^2)\cap L^\infty(0,\mathcal T; H^1\cap W^{1,q}),\quad
  \rho_t\in L^\infty(0,\mathcal T; L^2), \\
  &\rho u\in C([0,\mathcal T]; L^2), \quad u\in L^\infty(0,\mathcal T; D_0^1)\cap L^2(0,\mathcal T; D^2),\quad\sqrt\rho u_t\in L^2(0,
  \mathcal T; L^2),\\
  & \sqrt tu\in L^\infty(0,\mathcal T; D^2)\cap L^2(0,\mathcal T; D^{2,q}),\quad\sqrt tu_t\in L^2(0,\mathcal T; D_0^1).
\end{eqnarray*}
\end{theorem}

%\begin{remark}
%One can verify that the solution in Theorem \ref{keyth} admit the following additional regularities
%$$
%\sqrt t\sqrt\rho u_t\in L^\infty(0,\mathcal T; L^2), \quad \sqrt t\nabla^2u\in L^2(0,\mathcal T; L^q).
%$$
%\end{remark}

%\begin{theorem}\label{unique}
%Assume that the initial data $(\rho_0,u_0)$ satisfies
%$$0\leq \rho_0 \leq \bar{\rho},\quad \rho_0\in D_0^1,\quad \nabla \rho_0^{\gamma-\frac12} \in L^2\cap L^q, \quad u_0\in H^1_0,$$
%for some $\bar{\rho}\in (0,\infty)$ and $q\in (3,6)$.\\
%If $(\hat{\rho},\hat{u})$ and $(\check{\rho},\check{u})$ are two strong solutions to the problem (\ref{u})-(\ref{p}) corresponding to the same data $(\rho_0, u_0)$, then
%$$\hat{\rho}=\check{\rho}, \qquad \hat{u}=\check{u}.$$
%\end{theorem}

\begin{remark}
(i) Compared with the local well-posedness results established in \cite{CHOKIM04,CHOKIM06-1}, in Theorem \ref{keyth}, $u_0$ is not required to be in $D^2$ and we do not need any compatibility conditions on the initial data.

(ii) The same result as in \ref{keyth} also holds for the initial boundary value problem if imposing suitable boundary conditions on the velocity.
\end{remark}

\section{Lifespan estimate and some a priori estimates}

As preparations of proving the main result being carried out in the next section, the aim of this section is to
give the lifespan estimate and some a priori estimates, under the condition that the initial velocity $u_0\in D_0^1\cap D^2$
and some compatibility condition; however, it should be emphasized that all these estimates depend neither on $\|\nabla^2u_0\|_2$
nor on the compatibility condition.

We start with the following local existence and uniqueness result, which has been essentially proved in \cite{CHOKIM04,CHOKIM06-1,CHOKIM06-2}.

\begin{proposition}
  \label{Prop1}
Let $\rho_\infty\in[0,\infty)$ and $q\in(3,6)$ be fixed constants. Assume that the data $\rho_0$ and $u_0$ satisfy the regularity condition
$$
\rho_0\geq0, \quad \rho_0-\rho_\infty\in H^1\cap W^{1,q}, \quad u_0\in D_0^1\cap D^2,
$$
and the compatibility condition
$$
-\mu\Delta u_0-(\mu+\lambda)\nabla\text{div}\,u+\nabla P_0=\sqrt{\rho_0}g,
$$
for some $g\in L^2$, where $P_0=a\rho_0^\gamma$.

Then, there exists a small time $T_*>0$ and a unique strong solution $(\rho, u)$ to (\ref{u})--(\ref{rho}), subject to (\ref{u-bd1}), such that
\begin{eqnarray*}
  &\rho-\rho_\infty\in C([0,T_*]; H^1\cap W^{1,q}),\quad u\in C([0,T_*]; D_0^1\cap D^2)\cap L^2(0,T_*; D^{2,q}),\\
  & \rho_t\in C([0,T_*];L^2\cap L^q),\quad u_t\in L^2(0,T_*; D_0^1),\quad\sqrt\rho u_t\in L^\infty(0,T_*; L^2).
\end{eqnarray*}
\end{proposition}

As will be shown in this section, the existence time $T_*$ in the above proposition can be chosen
depending only on $\mu$, $\lambda$, $a$, $\gamma$, $q$, and the upper bound of
$$\Psi_0:=\|\rho_0\|_\infty+\|P_0-P_\infty\|_2+\|P_0\|_\infty+\|\nabla P_0\|_2+\|\nabla P_0\|_q+\|\nabla u_0\|_2,$$
with $P_\infty=a\rho_\infty^\gamma$, and, in particular, independent of $\|\nabla^2u_0\|_2$ and $\|g\|_2$.
The following quantity plays the key role in this section
\begin{eqnarray*}
    \Psi(t)&:=& (\|\rho\|_\infty+\|P-P_\infty\|_2+\|P\|_\infty+\|\nabla P\|_2+\|\nabla P\|_q\\
&&+\|\nabla u\|_2+\|\sqrt t\sqrt\rho u_t\|_2 )(t)+1.
\end{eqnarray*}

In the rest of this section, until the last proposition, we always assume $(\rho, u)$ is a solution to system (\ref{u})--(\ref{rho}), subject to (\ref{u-bd1}), on $\mathbb R^3\times(0,T)$, satisfying the regularities stated in Proposition \ref{Prop1}, with $T_*$ there replaced with $T$.

Throughout this section, except otherwise explicitly mentioned,
we denote by $C$ a generic constant depending only on $\mu$, $\lambda$, $a$, $\gamma$, and the upper bound of $\Psi_0.$

\begin{proposition}\label{Prop2}
The following estimates hold
  \begin{eqnarray*}
    \|\nabla^2u\|_2^2
  &\leq& C(\Psi^{10}+\Psi\|\sqrt\rho u_t\|_2^2), \\
   \|\sqrt\rho(u\cdot\nabla)u\|_2^2 &\leq& C(\Psi^9+\Psi^5\|\sqrt\rho u_t\|_2),\\
  \|\nabla^2u\|_q
   &\leq& C(\|\sqrt t\nabla u_t\|_2^2+\|\sqrt\rho u_t\|_2^2+t^{-\frac{5q-6}{4q}}+\Psi^{\alpha_1}), \quad q\in(3,6),
\end{eqnarray*}
with $\alpha_1:=\max\left\{12, \frac{(5q-6)^2}{2q(6-q)}\right\}$.
\end{proposition}

\begin{proof}
Applying the elliptic estimates to (\ref{u}) yields
$$
  \|\nabla^2u\|_2^2
   \leq C\|\rho\|_\infty(\|\sqrt\rho u_t\|_2^2+\|\sqrt\rho(u\cdot\nabla)u\|_2^2)+C\|\nabla P\|_2^2.
$$
By the H\"older and Sobolev inequality, one has
$$
  \|\sqrt\rho (u\cdot\nabla)u\|_2^2 \leq \|\rho\|_\infty\|u\|_6^2\|\nabla u\|_2\|\nabla u\|_6
  \leq C\Psi^4\|\nabla^2u\|_2.
$$
Substituting the above inequality into the previous one and using the Cauchy inequality, one gets
\begin{eqnarray*}
  \|\nabla^2u\|_2^2 &\leq &C(\Psi\|\sqrt\rho u_t\|_2^2+\Psi^5\|\nabla^2u\|_2+\Psi^2)\\
   &\leq& \frac12\|\nabla^2u\|_2^2+C(\Psi\|\sqrt\rho u_t\|_2^2+\Psi^{10}),
\end{eqnarray*}
that is
\begin{equation}\label{LI-1}
  \|\nabla^2u\|_2^2
  \leq C(\Psi^{10}+\Psi\|\sqrt\rho u_t\|_2^2),
\end{equation}
and, consequently,
$$
  \|\sqrt\rho(u\cdot\nabla)u\|_2^2 \leq C(\Psi^9+\Psi^5\|\sqrt\rho u_t\|_2),
$$
proving the first two conclusions.

It follows from the H\"older and Gagliardo-Nirenberg inequalities that
\begin{eqnarray}
  &&\|\rho(u\cdot\nabla)u\|_q \leq \|\rho\|_\infty\|u\|_{\frac{6q}{6-q}}\|\nabla u\|_6
  \nonumber\\
  &\leq& C\|\rho\|_\infty\|u\|_6^{\frac3q}\|\nabla u\|_6^{2-\frac3q}\leq C\|\rho\|_\infty\|\nabla u\|_2^{\frac3q}\|\nabla^2u\|_2^{\frac{2q-3}{q}},\label{LQ1}
\end{eqnarray}
from which, by the Young inequality and using (\ref{LI-1}),
one has
\begin{eqnarray*}
  \|\rho(u\cdot\nabla)u\|_q &\leq&C\Psi^{\frac{q+3}{q}}(\Psi^{10}+\Psi\|\sqrt\rho u_t\|_2^2
  )^{\frac{2q-3}{2q}} \\
&\leq& C\Psi^3(\Psi^{9}+\|\sqrt\rho u_t\|_2^2
  )^{\frac{2q-3}{2q}} \nonumber\\
 &\leq& C(\Psi^{2q}+\Psi^9+\|\sqrt\rho u_t\|_2^2)
 \leq C(\Psi^{12}+ \|\sqrt\rho u_t\|_2^2).
\end{eqnarray*}
It follows from the H\"older and Sobolev inequalities that
\begin{equation}
\label{LQ2}
  \|\rho u_t\|_q\leq\|\rho\|_\infty^{\frac{5q-6}{4q}}\|\sqrt\rho u_t\|_2^{\frac{6-q}{2q}}\|u_t\|_6^{\frac{3q-6}{2q}}
\leq C\|\rho\|_\infty^{\frac{5q-6}{4q}}\|\sqrt\rho u_t\|_2^{\frac{6-q}{2q}}\|\nabla u_t\|_2^{\frac{3q-6}{2q}},
\end{equation}
and further by the Young inequality that
\begin{eqnarray*}
  \|\rho u_t\|_q&\leq&C\Psi^{\frac{5q-6}{4q}}
  t^{-\frac{3q-6}{4q}}\|\sqrt t\nabla u_t\|_2^{\frac{3q-6}{2q}}\|\sqrt\rho u_t\|_2^{\frac{6-q}{2q}}\nonumber\\
  &\leq&C(\|\sqrt t\nabla u_t\|_2^2+\|\sqrt\rho u_t\|_2^2+t^{-\frac{5q-6}{4q}}+\Psi^{\frac{(5q-6)^2}{2q(6-q)}}).
\end{eqnarray*}
Thanks to the above two, and applying the elliptic estimates to (\ref{u}), one obtains
\begin{eqnarray*}
  \|\nabla^2u\|_q&\leq&C(\|\rho u_t\|_q+\|\rho(u\cdot\nabla)u\|_q+\|\nabla P\|_q)\nonumber\\
  &\leq&C(\|\sqrt t\nabla u_t\|_2^2+\|\sqrt\rho u_t\|_2^2+t^{-\frac{5q-6}{4q}}+\Psi^{\alpha_1}),
\end{eqnarray*}
proving the conclusion.
\end{proof}

\begin{proposition}
\label{Prop3}
The following estimate holds
$$
  \sup_{0\leq t\leq T}(\|\nabla u\|_2^2+\|P-P_\infty\|_2^2)
  +\int_0^T\|\sqrt\rho u_t\|_2^2 dt
  \leq C+C\int_0^T\Psi^{10}dt.
$$
\end{proposition}

\begin{proof}
Multiplying (\ref{u}) with $u_t$, it follows from integration by parts that
$$
  \frac12\frac{d}{dt}(\mu\|\nabla u\|_2^2+(\mu+\lambda)\|\text{div}\,u\|_2^2)+\|\sqrt\rho u_t\|_2^2
  =-\int(\rho(u\cdot\nabla)u+\nabla P)\cdot u_tdx.
$$
Integration by parts and noticing that
\begin{equation}\label{EQP}
  P_t+u\cdot\nabla P+\gamma\text{div}\,uP=0,
\end{equation}
one deduces
\begin{eqnarray*}
  -\int\nabla P\cdot u_tdx&=&\int (P-P_\infty)\text{div}\,u_tdx\\
  &=&\frac{d}{dt}\int (P-P_\infty)\text{div}\,udx-\int P_t\text{div}\,udx\\
  &=&\frac{d}{dt}\int (P-P_\infty)\text{div}\,udx+\int(u\cdot\nabla P+\gamma\text{div}\,uP)\text{div}\,udx.
\end{eqnarray*}
Therefore
\begin{eqnarray*}
\frac{d}{dt}\left( \mu\|\nabla u\|_2^2+(\mu+\lambda)\|\text{div}\,u\|_2^2-2\int (P-P_\infty)\text{div}\,udx\right)+\|\sqrt\rho u_t\|_2^2 \nonumber\\
  =\int(u\cdot\nabla P+\gamma\text{div}\,uP)\text{div}\,udx-\int \rho(u\cdot\nabla)u\cdot u_tdx
=:I_1+I_2.
\end{eqnarray*}
By the H\"older, Sobolev, and Young inequalities, and applying Proposition \ref{Prop2}, one has
\begin{eqnarray*}
  I_1&\leq&\|u\|_6\|\nabla P\|_2\|\text{div}\,u\|_3+\gamma\|\text{div}\,u\|_2^2\|P\|_\infty\\
  &\leq& C\|\nabla u\|_2\|\nabla P\|_2\|\nabla u\|_2^{\frac12}\|\nabla^2u\|_2^{\frac12} +C\Psi^3\\
  &\leq& C\Psi^{\frac52}(\Psi^{\frac52}+\Psi^{\frac14}\|\sqrt\rho u_t\|_2^{\frac12})+C\Psi^3 \\
  &\leq& \frac14\|\sqrt\rho u_t\|_2^2+C\Psi^5,
  \end{eqnarray*}
and
\begin{eqnarray*}
  I_2& \leq& \|\sqrt\rho u_t\|_2\|\sqrt\rho(u\cdot\nabla)u\|_2\\
   &\leq& C(\Psi^{\frac92}+\Psi^{\frac52}\|\sqrt\rho u_t\|_2^{\frac12})\|\sqrt\rho u_t\|_2\\
   &\leq&  \frac14\|\sqrt\rho u_t\|_2^2+C\Psi^{10}.
\end{eqnarray*}
Therefore
$$
\frac12\frac{d}{dt}\left( \mu\|\nabla u\|_2^2+(\mu+\lambda)\|\text{div}\,u\|_2^2-2\int (P-P_\infty)\text{div}\,udx\right)+\|\sqrt\rho u_t\|_2^2
\leq C\Psi^{10},
$$
from which, one obtains by the Cauchy inequality that
\begin{equation}\label{LI-2}
  \sup_{0\leq t\leq T}\|\nabla u\|_2^2+\int_0^T\|\sqrt\rho u_t\|_2^2 dt
  \leq C\left(1+ \sup_{0\leq t\leq T}\|P-P_\infty\|_2^2+\int_0^T\Psi^{10}dt\right).
\end{equation}

Multiplying (\ref{EQP}) with $P-P_\infty$, it follows from integration by parts and the Sobolev inequality that
\begin{eqnarray*}
  \frac{d}{dt}\|P-P_\infty\|_2^2&=&-(\gamma-\frac12)\int\text{div}\,u(P-P_\infty)^2dx-\gamma P_\infty\int\text{div}\,u(P-P_\infty)dx\\
  &\leq& C\|\nabla u\|_2\|{P-P_\infty}\|_2^{\frac12}\|\nabla P\|_2^{\frac32}+C\|\nabla u\|_2\|P-P_\infty\|_2\leq C\Psi^3,
\end{eqnarray*}
which gives
\begin{equation*}
  \sup_{0\leq t\leq T}\|P-P_\infty\|_2^2\leq C+C\int_0^T\Psi^3dt.
\end{equation*}
This, combined with (\ref{LI-2}), leads to the conclusion.
\end{proof}

The $t$-weighted estimate in the next proposition is the key to remove the compatibility condition on the initial data.

\begin{proposition}
  \label{Prop4}
The following estimate holds
$$
\sup_{0\leq t\leq T}
\|\sqrt t\sqrt\rho u_t\|_2^2+\int_0^T\|\sqrt t\nabla u_t\|_2^2dt\leq C+C\int_0^T\Psi^{16}dt.
$$
\end{proposition}

\begin{proof}
Differentiating (\ref{u}) in $t$ and using (\ref{rho}) yield
\begin{eqnarray*}
  \rho (\partial_t u_{t} + (u\cdot \nabla) u_t)-\mu \Delta u_t -(\lambda +\mu )\nabla \text{div}\, u_t \\
  =  -\nabla P_t+\text{div}\,(\rho u)(u_t + (u\cdot \nabla) u) -\rho( u_t\cdot \nabla) u.
\end{eqnarray*}
Multiplying it by $u_t$, integrating by parts over $\mathbb{R}^3$ and then using the continuity equation (\ref{rho}), one has
\begin{eqnarray*}
    &&\frac12 \frac{d}{dt} \|\sqrt \rho u_t\|_2^2+\mu\|\nabla u_t\|_2^2 + (\lambda +\mu)\|\text{div} \, u_t\|_2^2  \\
     &= & \int P_t \text{div} \, u_tdx +\int\text{div}\,( \rho u ) |u_t|^2dx+\int\text{div}\,( \rho u ) (u\cdot\nabla) u\cdot u_tdx\\
     &&-\int\rho (u_t\cdot \nabla u)\cdot u_t dx=: II_1+II_2+II_3+II_4.
\end{eqnarray*}
Recalling (\ref{EQP}) and using the Sobolev and Young inequalities, one deduces
\begin{eqnarray*}
  II_1&=&-\int (\gamma \text{div}\, u P +u\cdot \nabla P) \text{div} \, u_t dx\\
  &\leq& C \left(\|P\|_{\infty}||\nabla u||_{2}||\nabla u_t||_{2}+ ||u||_6 ||\nabla P||_{3}||\nabla u_t||_{2}\right)\\
  &\leq & C\left(\Psi^4+||\nabla u||_{2}^2 ||\nabla P||_{L^2\cap L^q}^2\right) + \frac\mu8 ||\nabla u_t||_{2}^2\\
  &\leq& C\Psi^4 + \frac\mu8 ||\nabla u_t||_{2}^2.
\end{eqnarray*}
Integrating by parts, using the H\"{o}lder, Sobolev and Young inequalities, and applying Proposition \ref{Prop2}, we have
\begin{eqnarray*}
  II_2&=& -\int_{\mathbb{R}^3} \rho u \cdot \nabla  |u_t|^2 dx\\
   &\leq&  C ||\rho||_{\infty}^{\frac12}||u||_{6}||\sqrt{\rho}u_t||_{2}^{\frac12}||\sqrt{\rho}u_t||_{6}^{\frac12}||\nabla u_t||_{2}\\
  &\leq& C ||\rho||_{\infty}^{\frac34}||\nabla u||_{2}||\sqrt{\rho}u_t||_{2}^{\frac12}||\nabla u_t||_{2}^{\frac32}\\
   &\leq&   C\Psi^7 ||\sqrt{\rho}u_t||_{2}^{2} +\frac\mu8 ||\nabla u_t||_{2}^2,\\
  II_3&\leq&\int_{\mathbb{R}^3} \rho|u|(|\nabla u|^2|u_t|+|u||\nabla^2u||u_t|+|u||\nabla u||\nabla u_t|)dx \\
  &\leq& C||\rho||_{\infty} (||u||_{6}||\nabla u||_{3}^2||u_t||_{6}+ ||u||_{6}^2||\nabla^2 u||_{2}||u_t||_{6}\\
&&+
  ||u||_{6}^2||\nabla u||_{6}||\nabla u_t||_2)\\
  &\leq& C  ||\rho||_{\infty}||\nabla u||_{2}^{2}||\nabla^2 u||_{2}||\nabla u_t||_{2}\\
  &\leq&  C \Psi^6 ||\nabla^2 u||_{2}^2+\frac\mu8 ||\nabla u_t||_{L^2}^2\\
   &\leq&C(\Psi^{16}+\Psi^7 ||\sqrt\rho u_t||_{2}^2)+\frac\mu8 ||\nabla u_t||_{L^2}^2,
\end{eqnarray*}
and
\begin{eqnarray*}
  II_4&\leq & \int_{\mathbb{R}^3} \rho |u_t|^2 |\nabla u| dx\\
    &\leq&  C ||\rho||_{\infty} ^{\frac12}||\nabla u||_{2}||\sqrt{\rho}u_t||_{2}^{\frac12}||\sqrt{\rho}u_t||_{6}^{\frac12}||u_t||_{6}\\
  &\leq& C ||\rho||_{\infty}^{\frac34}||\nabla u||_{2}||\sqrt{\rho}u_t||_{2}^{\frac12}||\nabla u_t||_{2}^{\frac32}\\
   &\leq&   C\Psi^7 ||\sqrt{\rho}u_t||_{2}^{2} +\frac\mu8 ||\nabla u_t||_{2}^2.
\end{eqnarray*}
Therefore, we have
$$
  \frac{d}{dt} \|\sqrt \rho u_t\|_2^2+\mu\|\nabla u_t\|_2^2
  \leq C(\Psi^{16}+\Psi^7 ||\sqrt\rho u_t||_{2}^2),
$$
which, multiplied by $t$, gives
\begin{eqnarray*}
  \frac{d}{dt}\|\sqrt t\sqrt\rho u_t\|_2^2+\mu\|\sqrt t\nabla u_t\|_2^2&\leq& C(\Psi^{16}+\Psi^7 ||\sqrt t\sqrt\rho u_t||_{2}^2+\|\sqrt\rho u_t\|_2^2)\\
  &\leq& C(\Psi^{16}+\|\sqrt\rho u_t\|_2^2).
\end{eqnarray*}
Integrating this in $t$ and applying Proposition \ref{Prop3}, the conclusion follows.
\end{proof}

\begin{proposition}
\label{Prop5}
  The following estimate holds
  $$
  \int_0^T(\|\nabla u\|_\infty+\|\nabla^2u\|_q)dt\leq C+C\int_0^T\Psi^{\alpha_2}dt,
  $$
  with $\alpha_2:=\max\left\{16,\alpha_1\right\}=\max\left\{16,\frac{(5q-6)^2}{2q(6-q)}\right\}$.
\end{proposition}

\begin{proof}
Noticing that $t^{-\frac{5q-6}{4q}}\in(0,1)$, for $q\in(3,6)$, and recalling the following estimate by Proposition \ref{Prop2}
  $$
  \|\nabla^2u\|_q
     \leq  C(\|\sqrt t\nabla u_t\|_2^2+\|\sqrt\rho u_t\|_2^2+t^{-\frac{5q-6}{4q}}+\Psi^{\alpha_1}),
  $$
  it follows from the Gagliardo-Nirenberg and Young inequalities and Propositions \ref{Prop3} and \ref{Prop4} that
  \begin{eqnarray*}
    \int_0^T\|\nabla u\|_\infty dt&\leq&C\int_0^T\|\nabla u\|_2^{1-\theta}\|\nabla^2u\|_q^\theta dt \\
&\leq& C\int_0^T(\|\nabla u\|_2+\|\nabla^2u\|_q)dt \\
    &\leq&  C\int_0^T(\|\sqrt t\nabla u_t\|_2^2+\|\sqrt\rho u_t\|_2^2+t^{-\frac{5q-6}{4q}}+\Psi^{\alpha_1})dt\\
&\leq& C+C\int_0^T\Psi^{\alpha_2}dt,
  \end{eqnarray*}
where $\theta=\frac{3q}{5q-6}\in(0,1)$, proving the conclusion.
\end{proof}

\begin{proposition}
\label{Prop6}
The following estimate holds
$$
\sup_{0\leq t\leq T}(\|\rho\|_\infty+\|P\|_\infty)\leq C\exp\left(C\int_0^T\Psi^{\alpha_2}dt\right),
$$
where $\alpha_2$ is the number in Proposition \ref{Prop5}.
\end{proposition}

\begin{proof}
Define $X(t;x)$ as
\begin{equation*}\label{X}
  \left\{
    \begin{array}{ll}
      \frac{d}{dt}X(t;x)=u(X(t;x),t), & \\
      X(0;x)=x. &
    \end{array}
  \right.
\end{equation*}
One can show that for any $t\in(0,T)$, and for any $y\in\mathbb R^3$, there is a unique $x\in\mathbb R^3$, such that $X(t;x)=y$, and, in particular,
$X(t;\mathbb R^3)=\mathbb R^3$; in fact, to show this, it suffices to consider the backward problem
$\frac{d}{dt}Z(t)=u(Z(t),t), X(T;x)=y.$
Then, by (\ref{rho}), it has
\begin{eqnarray*}
\frac{d}{dt}\rho(X(t;x),t)&=&\partial_t\rho(X(t;x),t)+u(X(t;x),t)\cdot \nabla \rho(X(t;x),t)\\
&=&-\text{div}\, u(X(t;x),t) \rho(X(t;x),t),
\end{eqnarray*}
and, thus,
\begin{equation}
\label{exprho}
 \rho(X(t;x),t)=\rho_0(x)\exp\left(-\int_0^t\text{div}\, u(X(\tau;x),\tau)d\tau\right).
\end{equation}
Therefore,
\begin{eqnarray*}
\|\rho\|_\infty(t)&=&\|\rho(X(t;x),t)\|_\infty(t)\\
&\leq&\|\rho_0\|_\infty\exp\left(\int_0^T\|\nabla u\|_\infty dt\right),
\end{eqnarray*}
and the conclusion follows by applying Proposition \ref{Prop4}.
\end{proof}

\begin{proposition}
\label{Prop7}
The following estimate holds
  $$
 \sup_{0\leq t\leq T}(\|\nabla P\|_2+||\nabla P||_q)
 \leq C \exp\left(C\int_0^T\Psi^{\alpha_2}dt\right),\quad q\in(3,6),
$$
where $\alpha_2$ is the number in Proposition \ref{Prop5}.
\end{proposition}

\begin{proof}
From (\ref{EQP}), one has
\begin{equation*}\label{grade p}
  \partial_t \nabla P +\gamma \text{div}\,  u  \nabla P +\gamma P  \nabla \text{div}\, u+(u\cdot\nabla)\nabla P + \nabla P\nabla u  =0.
\end{equation*}
Multiplying the above by $|\nabla P|^{p-2}\nabla P $, integrating over $\mathbb{R}^3$, one has
\begin{eqnarray*}
  \frac{d}{dt}\|\nabla P\|_p^p  \leq   C  (||\nabla u||_{\infty}||\nabla P||_{p}^p +||P||_{\infty}||\nabla ^2 u||_{p}||\nabla P||_{p}^{p-1}),
\end{eqnarray*}
which gives
$$ \frac{d}{dt}||\nabla P||_{p} \leq   C ( ||\nabla u||_{\infty}||\nabla P||_{p} +||P||_{\infty}||\nabla ^2 u||_{p}).$$
By the Gronwall inequality, one has
\begin{equation}\label{grade-P}
  \begin{split}
     \sup_{0\leq t\leq T}||\nabla P||_{p} \,\leq&\,
     C\left(\|\nabla P_0\|_p+\int_0^T\|P\|_\infty\|\nabla^2u\|_pdt\right)\exp\left(C\int_0^T\|\nabla u\|_\infty dt\right)\\
     \leq&C\left(1+\int_0^T\|\nabla^2u\|_pdt\right)\exp\left(C\int_0^T\Psi^{\alpha_2}dt\right).
  \end{split}
\end{equation}
Thanks to the above, it follows from Proposition \ref{Prop5} and Proposition \ref{Prop6} that
\begin{eqnarray*}
 \sup_{0\leq t\leq T}||\nabla P||_q&\leq& C\left(1+\int_0^T\Psi^{\alpha_2}dt\right)\exp\left(C\int_0^T\Psi^{\alpha_2}dt\right)\\
 &\leq& C \exp\left(C\int_0^T\Psi^{\alpha_2}dt\right),
\end{eqnarray*}
where we have used the fact that $e^z\geq 1+z$ for $z\geq0$. By Proposition \ref{Prop2} and Proposition \ref{Prop3}, it follows from
(\ref{grade-P}) and the Cauchy inequality that
\begin{eqnarray*}
  \sup_{0\leq t\leq T}||\nabla P||_2&\leq&C\left[1+\int_0^T(\Psi^5+\Psi^{\frac12}\|\sqrt\rho u_t\|_2)dt\right]\exp\left(C\int_0^T\Psi^{\alpha_2}dt\right)\\
  &\leq&C\left[1+\int_0^T(\Psi^5+ \|\sqrt\rho u_t\|_2^2)dt\right]\exp\left(C\int_0^T\Psi^{\alpha_2}dt\right)\\
  &\leq&C\left(1+\int_0^T\Psi^{10}dt\right)\exp\left(C\int_0^T\Psi^{\alpha_2}dt\right) \\
&\leq& C\exp\left(C\int_0^T\Psi^{\alpha_2}dt\right).
\end{eqnarray*}
This proves the conclusion.
\end{proof}

\begin{proposition}
  \label{Prop8}
The following estimates hold
$$
\sup_{0\leq t\leq T}(\|\rho-\rho_\infty\|_2+\|\nabla\rho\|_2+\|\nabla\rho\|_q)\leq C \exp\left(C\int_0^T\Psi^{\alpha_2}dt\right),\quad q\in(3,6),
$$
with constant $C$ depending also on $\|\rho_0-\rho_\infty\|_2+\|\nabla\rho_0\|_2+\|\nabla\rho_0\|_q$,
and
$$
\sup_{0\leq t\leq T}\|\sqrt t\nabla^2u\|_2^2\leq C\sup_{0\leq t\leq T}(\Psi^{10}+\Psi\|\sqrt t\sqrt\rho u_t\|_2^2),
$$
where $\alpha_2$ is the number in Proposition \ref{Prop5}.
\end{proposition}

\begin{proof}
The estimate of $\|\rho-\rho_\infty\|_2$ follows in the same way as that for $\|P-P_\infty\|_2$ in Proposition \ref{Prop3}, while those for
$\|\nabla\rho\|_2$ and $\|\nabla\rho\|_q$ can be proved similarly as in Proposition \ref{Prop7}. The conclusion for $\|\sqrt t\nabla^2u\|_2^2$ follows from combining Propositions \ref{Prop2} and \ref{Prop4}.
\end{proof}

We end up this section with the following proposition on the lifespan estimate and a priori estimates.

\begin{proposition}
  \label{Prop9}
Assume in addition to the conditions in Proposition \ref{Prop1} that $\rho\geq\underline\rho$ for some positive constant $\underline\rho$.

Then, there are two positive constants $\mathcal T$ and $C$
depending only on $\mu$, $\lambda$, $a$, $\gamma$, $q$, and the upper bound of
$\psi_0:=\|\rho_0\|_\infty+\|\rho_0-\rho_\infty\|_2+\|\nabla\rho_0\|_{L^2\cap L^q}+\|\nabla u_0\|_2,$
and, in particular, independent of $\underline\rho$ and
$\|\nabla^2u_0\|_2$, such that system (\ref{u})--(\ref{rho}), subject to (\ref{u-bd1}), has a unique solution $(\rho, u)$ on $\mathbb R^3\times(0,\mathcal T)$, enjoying the regularities stated in Proposition \ref{Prop1}, with $T_*$ there replaced by $\mathcal T$, and the following a priori estimates
\begin{eqnarray*}
  \sup_{0\leq t\leq\mathcal T}(\|\rho\|_\infty+\|\rho-\rho_\infty\|_2+\|\nabla\rho\|_2+\|\nabla\rho\|_q+\|\rho_t\|_2)&\leq& C, \\
  \sup_{0\leq t\leq\mathcal T} \|\nabla u\|_2^2 +\int_0^\mathcal T(\|\nabla^2u\|_2^2+\|\sqrt\rho u_t\|_2^2)dt&\leq& C,  \\
  \sup_{0\leq t\leq\mathcal T}(\|\sqrt t\sqrt\rho u_t\|_2^2+\|\sqrt t\nabla^2u\|_2^2)
  +\int_0^\mathcal T(\|\sqrt t\nabla u_t\|_2^2+\|\sqrt t\nabla^2u\|_q^2)dt&\leq& C.
\end{eqnarray*}
\end{proposition}

\begin{proof}
Define the maximal time $T_\text{max}$ as
$$
T_\text{max}:=\max\{T\in\mathscr T\},
$$
where
\begin{eqnarray*}
\mathscr T&:=&\{T\in[T_*,\infty)~|~\text{There is a solution }(\rho, u)\mbox{ in the class }\mathscr X_T\\
&& \text{ to system }(\ref{u})-(\ref{rho}),\text{ subject to } (\ref{u-bd1}), \text{ on }\mathbb R^3\times(0,T)\},
\end{eqnarray*}
where $\mathscr X_T$ is the class of $(\rho, u)$ enjoying the regularities as stated in Proposition \ref{Prop1}, with $T_*$ there replaced with $T$.
By Proposition \ref{Prop1}, it is clear that $T^\text{max}$ is well defined and $T^\text{max}\geq T_*$.
Moreover, by the uniqueness result, see the proof of the uniqueness part of Theorem \ref{keyth} in the next section,
one can easily show that any two solutions $(\bar\rho, \bar u)$ and $(\tilde\rho,\tilde u)$ to system (\ref{u})--(\ref{rho}), subject to
(\ref{u-bd1}), on $\mathbb R^3\times(0,\bar T)$ and on $\mathbb R^3\times(0,\tilde T)$, respectively, coincide on $\mathbb R^3\times(0,\min\{\bar T,
\tilde T\})$. Choose $T_k\in\mathscr T$ with $T_k\uparrow T^\text{max}$ as $k\uparrow\infty$. By definition of $\mathscr T$, there is a solution
$(\rho_k, u_k)$ to system (\ref{u})--(\ref{rho}), subject to (\ref{u-bd1}), on $\mathbb R^3\times(0, T_k)$. Define $(\rho, u)$ on
$\mathbb R^3\times(0,T^\text{max})$ as
$$
(\rho, u)(x,t)=(\rho_k,u_k)(x,t), \quad x\in\mathbb R^3, t\in(0,T_k), k=1,2,\cdots.
$$
Applying the uniqueness result again, the definition of $(\rho, u)$ is independent of the choice of the sequence $\{T_k\}_{k=1}^\infty$. By the construction of $(\rho, u)$, one can verify that $(\rho, u)$ is a solution to (\ref{u})--(\ref{rho}), subject to (\ref{u-bd1}), on $\mathbb R^3\times(0, T^\text{max})$, and $(\rho, u)\in\mathscr X_T$, for any $T\in(0,T^\text{max})$.

By Propositions \ref{Prop3}, \ref{Prop4}, \ref{Prop6}, and \ref{Prop7}, it is clear
$$
\Psi(t)\leq C_m\exp\left(C_m\int_0^t\Psi^{\alpha_2}d\tau\right),\quad t\in(0,T^{\text{max}}),
$$
where $C_m$ is a positive constant depending only on $\mu$, $\lambda$, $a$, $\gamma$, $q$, and the upper bound of
$\psi_0$. Here we have used the fact that $\Psi_0$ can be controlled by $\psi_0$.
One can easily derive from the above inequality that
\begin{equation}
\label{ESTD}
{\Psi(t)\leq2^{\frac{1}{\alpha_2}} C_m,\quad \forall t\in\left(0,\min\left\{T^{\text{max}}, (2^{\alpha_2}C_m^{\alpha_2+1})^{-1}\right\}\right).}
\end{equation}
Thanks to the above estimate, one can get by applying Propositions \ref{Prop5} and \ref{Prop8} that
\begin{equation}
\label{ESTE}
  (\|\rho-\rho_\infty\|_2+\|\nabla\rho\|_2+\|\nabla\rho\|_q+\|\sqrt t\nabla^2u\|_2^2)(t)+\int_0^t\|\nabla u\|_\infty d\tau\leq C,
\end{equation}
 for any $t\in\left(0,\min\left\{T^{\text{max}}, {(2^{\alpha_2}C_m^{\alpha_2+1})^{-1}}\right\}\right)$, and for a positive constant $C$ depending only on $\mu$, $\lambda$, $a$, $\gamma$, $q$, and the upper bound of
$\psi_0$. Thanks to (\ref{ESTD})--(\ref{ESTE}) and using (\ref{rho}) one can further obtain
\begin{eqnarray}
  \nonumber\|\rho_t\|_2&=&\|u\cdot\nabla\rho+\text{div}\,u\rho\|_2\nonumber\\
&\leq&\|u\|_6\|\nabla\rho\|_3+\|\nabla u\|_2\|\rho\|_\infty\nonumber\\
  &\leq&C(1+\|\nabla\rho\|_{L^2\cap L^q})\|\nabla u\|_2\leq C, \label{ESTF}
\end{eqnarray}
for any $0<t<\min\{T^{\text{max}}, {(2^{\alpha_2}C_m^{\alpha_2+1})^{-1}}\}$. Using the estimate $\int_0^t\|\nabla u\|_\infty d\tau\leq C$ in (\ref{ESTE}) and recalling (\ref{exprho}), it is clear that
\begin{equation}
\label{ESTG}
  \rho(x,t)\geq C\underline\rho, \quad x\in\mathbb R^3, \quad 0<t<\min\left\{T^{\text{max}}, {(2^{\alpha_2}C_m^{\alpha_2+1})^{-1}}\right\}.
\end{equation}

We claim that $T^\text{max}>{(2^{\alpha_2}C_m^{\alpha_2+1})^{-1}}$. Assume in contradiction that this does not hold. Then, all the estimates in (\ref{ESTD})--(\ref{ESTG}) hold for any $t\in(0,T^\text{max})$. Estimates (\ref{ESTD})--(\ref{ESTG}),
holding on time interval $(0,T^\text{max})$, guarantee that $(\rho, u)(\,\cdot\,, t)$
can be uniquely extended to time $T^\text{max}$, with $(\rho, u)(\,\cdot\,, T^\text{max})$ defined as the limit of $(\rho, u)(\,\cdot\,, t)$
as $t\uparrow T^\text{max}$, and that
$$(\rho-\rho_\infty)(\,\cdot\,,T^\text{max})\in H^1\cap W^{1,q},\quad u(\,\cdot\,,T^\text{max})\in D_0^1\cap D^2.$$
Thanks to this and recalling (\ref{ESTG}), it is clear that the compatibility condition holds at time $T^\text{max}$. Therefore, by the local
well-posedness result, i.e., Proposition \ref{Prop1}, one can further extend solution $(\rho, u)$ beyond the time $T^\text{max}$, which contradicts to the definition of $T^\text{max}$. This contradiction proves the claim that $T^\text{max}>  {(2^{\alpha_2}C_m^{\alpha_2+1})^{-1}}$.
As a result, one obtains a solution $(\rho, u)$ on time interval $(0,  {(2^{\alpha_2}C_m^{\alpha_2+1})^{-1}})$ satisfying all the a priori estimates in (\ref{ESTD})--(\ref{ESTF}), except $\int_0^\mathcal T\|\sqrt t\nabla^2u\|_q^2dt\leq C$, on the same time interval.

It remains to verify $\int_0^\mathcal T\|\sqrt t\nabla^2u\|_q^2dt\leq C$. To this end, recalling (\ref{LQ1}) and (\ref{LQ2}), it follows
from the elliptic estimate, the estimates just obtained, and the Young inequality that
\begin{eqnarray*}
  \|\nabla^2u\|_q&\leq&C(\|\rho u_t\|_q+\|\rho(u\cdot\nabla)u\|_q+\|\nabla P\|_q)\nonumber\\
  &\leq&C( \|\rho\|_\infty^{\frac{5q-6}{4q}}\|\sqrt\rho u_t\|_2^{\frac{6-q}{2q}}\|\nabla u_t\|_2^{\frac{3q-6}{2q}}+
  \|\rho\|_\infty\|\nabla u\|_2^{\frac3q}\|\nabla^2u\|_2^{\frac{2q-3}{q}}+\|\nabla\rho\|_q)\\
&\leq&C(1+\|\nabla^2u\|_2^2+\|\sqrt\rho u_t\|_2+\|\nabla u_t\|_2),
\end{eqnarray*}
and further that
$$
 \int_0^\mathcal T\|\sqrt t\nabla^2u\|_q^2dt\leq C\int_0^\mathcal T(1+\|\nabla^2u\|_2^2\|\sqrt t\nabla^2u\|_2^2
+\|\sqrt\rho u_t\|_2^2+\|\sqrt t\nabla u_t\|_2^2)dt\leq C,
$$
proving the conclusion.
\end{proof}

\section{Proof of Theorem \ref{keyth}}

This section is devoted to proving Theorem \ref{keyth}.

The following lemma, proved in \cite{LJK17}, will be used in proving the uniqueness.

\begin{lemma}
\label{uni-lem}
Given a positive time $T$ and nonnegative functions $f,g,G$ on $[0,T]$, with $f$ and $g$ being absolutely continuous on $[0,T]$. Suppose that
\begin{equation*}
\left\{
  \begin{array}{ll}
    \frac{d}{dt} f(t)\leq \delta (t)f(t) +A \sqrt{G(t)}, & \\
    \frac{d}{dt} g(t) +G(t) \leq \alpha(t)g(t) +\beta(t)f^2(t) , &  \\
    f(0)=0, &
  \end{array}
\right.
\end{equation*}
where $\alpha$, $\beta$ and $\delta$ are three nonnegative functions, with $\alpha,\delta,t\beta\in L^1((0,T))$.

Then, then following estimates hold
\begin{eqnarray*}
f(t)& \leq& AB \sqrt{tg(0)} \exp\left(\frac{1}{2} \int_0^t ( \alpha(s)+A^2B^2s\beta(s))ds\right), \\
g(t) +\int_0^tG(s)ds&\leq& g(0) \exp\left(\int_0^t (\alpha(s) +A^2B^2s\beta(s))ds\right),
\end{eqnarray*}
where $B=1+e^{\int_0^T\delta(\tau)d\tau}$. In particular, if $g(0)=0$, then $f\equiv g\equiv0$ on $(0,T)$.
\end{lemma}

We are now ready to prove Theorem \ref{keyth}.

\begin{proof}[\textbf{Proof of Theorem \ref{keyth}}]
We first prove the uniqueness and then the existence.

\textbf{Uniqueness:}
Let $(\check{\rho}, \check{u})$, $(\hat{\rho}, \hat{u})$ be two solutions of system (\ref{u})-(\ref{rho}), subject to (\ref{u-bd1}),
on $\mathbb R^3\times(0,T)$, satisfying the regularities stated in the theorem.
For $u\in\{\hat u,\check u\}$, by the Gagliardo-Nirenberg and H\"older inequalities, one has
\begin{eqnarray*}
  \int_0^T\|\nabla u\|_\infty dt&\leq&C\int_0^T\|\nabla u\|_2^{1-\theta}\|\nabla^2u\|_q^\theta dt\leq C\int_0^T\|\sqrt t\nabla^2u\|_q^\theta
t^{-\frac\theta2}dt\\
&\leq&C\left(\int_0^T\|\sqrt t\nabla^2u\|_q^2dt\right)^{\frac\theta2}\left(\int_0^Tt^{-\frac{\theta}{2-\theta}}dt\right)^{1-\frac\theta2}<\infty,
\end{eqnarray*}
where $\theta=\frac{3q}{5q-6}\in(0,1)$. Therefore, $\nabla\hat u, \nabla \check u\in L^1(0, T; L^\infty)$.

Denote
$$\sigma=\check{\rho}-\hat{\rho},\quad W=\check{P}-\hat{P},\quad v=\check{u}-\hat{u}.$$ Then, straightforward calculations show
\begin{eqnarray}
  \sigma_t+v\cdot \nabla\hat{\rho} +\check{u}\cdot \nabla \sigma+\text{div}\,\check{u} \sigma +\text{div}\,v\hat{\rho}=0,\label{DEQRHO}\\
  \hat{\rho}(v_t +\hat{u}\cdot\nabla v) -\mu\Delta v-(\lambda+\mu)\nabla \text{div}\, v+ \nabla W
      =  -\sigma\check{u}_t-\sigma\check{u}\cdot\nabla \check{u} -\hat{\rho}v\cdot\nabla \check{u},\label{DEQU}\\
W_t+v\cdot \nabla\hat{P} +\check{u}\cdot \nabla W+\gamma \text{div}\,\check{u} W +\gamma \text{div}\,v\hat{P}=0.\label{DEQW}
\end{eqnarray}

Testing (\ref{DEQRHO}) with $\sigma$ and using the H\"{o}lder inequality, we have
\begin{eqnarray*}
     \frac{d}{dt} \int  |\sigma|^{2}dx &\leq& C\int  (|v\cdot\nabla \hat{\rho}|\,|\sigma| +|\text{div}\check{u}||\sigma|^{2}+| \text{div}\,v\hat{\rho}|\,|\sigma|) \\
      &\leq& C(||\nabla \hat{\rho}||_{3}||\sigma||_{2}||v||_{6} +||\nabla \check{u}||_{\infty}||\sigma||_{{ 2}}^{ 2}) +C||\hat{\rho}||_{\infty}||\nabla v||_{2}||\sigma||_{{2}} \\
&\leq&C||\nabla \check{u}||_{\infty}||\sigma||_{{2}}^2+ C(\|\hat\rho\|_\infty+\|\nabla\hat\rho\|_3) ||\nabla v||_{2}||\sigma||_{{2}},
\end{eqnarray*}
and, thus,
\begin{equation}
  \label{USIGMA2}
 \frac{d}{dt}\|\sigma\|_2\leq C||\nabla \check{u}||_{\infty}||\sigma||_{{2}}+ C(\|\hat\rho\|_\infty+\|\nabla\hat\rho\|_3) ||\nabla v||_{2}.
\end{equation}
Similarly, by testing (\ref{DEQW}) with $W$ yields
\begin{equation}\label{UW}
  \frac{d}{dt}||W||_{2}\leq C||\nabla \check{u}||_{\infty}||W||_{2}+C (||\nabla \hat{P}||_{3}+||\hat{P}||_{\infty})||\nabla v||_{2}.
\end{equation}
Testing (\ref{DEQU}) with $v$ and using the H\"{o}lder and Young inequalities, we have
\begin{eqnarray}
\nonumber&&\frac12 \frac{d}{dt}\int \hat{\rho}|v|^2 dx +\int  [\mu|\nabla v|^2+(\lambda+\mu) (\text{div}\, v)^2]dx\\
\label{DENGYU}&\leq& \int (|W||\nabla v| + |\sigma|| \check{u}_t||v|+|\sigma||\check{u}||\nabla \check{u}||v|+|\hat{\rho}(v\cdot\nabla \check{u})\cdot v|)dx=:RHS.
\end{eqnarray}

We proceed the proof separately for the cases $\rho_\infty=0$ and $\rho_\infty>0$.

\textbf{Case I: }$\rho_\infty=0$.

By the H\"older, Sobolev, and Young inequalities, we can control $RHS$ as
\begin{eqnarray*}
  RHS&\leq&\|W\|_2\|\nabla v\|_2+\|\sigma\|_{\frac32}\|\check u_t\|_6\|v\|_6\\
&&+\|\sigma\|_2\|\check u\|_6\|\nabla\check u\|_6\|v\|_6
+\|\nabla\check u\|_\infty\|\sqrt{\hat\rho}v\|_2^2\nonumber\\
&\leq&C(\|W\|_2\|\nabla v\|_2+\|\sigma\|_{\frac32}\|\nabla\check u_t\|_2\|\nabla v\|_2\nonumber\\
&&+\|\sigma\|_2\|\nabla\check  u\|_2\|\nabla^2\check u\|_2\|\nabla v\|_2
+\|\nabla\check u\|_\infty\|\sqrt{\hat\rho}v\|_2^2)\nonumber\\
&\leq&\frac\mu2\|\nabla v\|_2^2+C(\|W\|_2^2+\|\nabla\check u_t\|_2^2\|\sigma\|_{\frac32}^2\\
&&+\|\nabla\check  u\|_2^2\|\nabla^2\check u\|_2^2\|\sigma\|_2^2+\|\nabla\check u\|_\infty\|\sqrt{\hat\rho}v\|_2^2),
\end{eqnarray*}
which plugged into (\ref{DENGYU}) leads to
\begin{eqnarray}
\nonumber\frac{d}{dt}\|\sqrt{\hat\rho}v\|_2^2+\mu\|\nabla v\|_2^2&\leq  & C\|\nabla\check u\|_\infty\|\sqrt{\hat\rho}v\|_2^2
+C(1+\|\nabla\check u_t\|_2^2+\|\nabla\check u\|_2^2|\nabla^2\check u\|_2^2)\\
&&\times(\|W\|_2^2+\|\sigma\|_2^2+\|\sigma\|_{\frac32}^2).\label{UU1}
\end{eqnarray}
The appearance of $\|\sigma\|_{\frac32}$ in the above inequality requires the energy estimate for $\|\sigma\|_{\frac32}$ given in the below.

Testing (\ref{DEQRHO}) with $sign(\sigma)|\sigma|^{\frac12}$ and using the H\"older and Sobolev inequalities that
\begin{eqnarray*}
     \frac{d}{dt} \int  |\sigma|^{\frac32}dx &\leq& C\int  (|v\cdot\nabla \hat{\rho}|\,|\sigma|^\frac12+|\text{div}\check{u}||\sigma|^{\frac32}+| \text{div}\,v\hat{\rho}|\,|\sigma|^\frac12) \\
      &\leq& C(||\nabla \hat{\rho}||_{2}||\sigma||_{\frac32}^{\frac12}||v||_{6} +||\nabla \check{u}||_{\infty}||\sigma||_{{\frac32}}^{\frac32}) +C||\hat{\rho}||_{6}||\nabla v||_{2}||\sigma||_{{\frac32}}^\frac12\\
&\leq&C||\nabla \check{u}||_{\infty}||\sigma||_{{\frac32}}^\frac32+ C ||\nabla \hat{\rho}||_{2}  ||\nabla v||_{2}||\sigma||_{{\frac32}}^\frac12,
\end{eqnarray*}
which gives
\begin{equation}
  \frac{d}{dt}\|\sigma\|_{\frac32}\leq C||\nabla \check{u}||_{\infty}||\sigma||_{{\frac32}} + C ||\nabla \hat{\rho}||_{2}  ||\nabla v||_{2}.
\label{USIGMA32}
\end{equation}

Denote
\begin{eqnarray*}
&&f_1(t)=(\|\sigma\|_{\frac32}+\|\sigma\|_2+\|W\|_2)(t),\quad g_1(t)=\|\sqrt{\hat\rho}v\|_2^2(t), \quad G_1(t)=\mu\|\nabla v\|_2^2(t), \\
&&\delta_1(t)=C\|\nabla\check u\|_\infty(t),\quad A_1=C\sup_{0\leq t\leq T}(\|\hat\rho\|_\infty+\|\hat P\|_\infty+\|\nabla\hat\rho\|_{L^2\cap L^3}+\|\nabla\hat P\|_3)(t),\\
&&\alpha_1(t)=C\|\nabla\check u\|_\infty(t),\quad \beta_1(t)=C(1+\|\nabla\check u_t\|_2^2+\|\nabla\check u\|_2^2\|\nabla^2\check u\|_2^2)(t),
\end{eqnarray*}
then, it follows from (\ref{USIGMA2}), (\ref{UW}), (\ref{UU1}), and (\ref{USIGMA32}) that
\begin{equation*}
\left\{
  \begin{array}{ll}
    \frac{d}{dt} f_1(t)\leq \delta_1 (t)f_1(t) +A_1 \sqrt{G_1(t)}, & \\
    \frac{d}{dt} g_1(t) +G_1(t) \leq \alpha_1(t)g_1(t) +\beta_1(t)f_1^2(t) , &  \\
    f_1(0)=0.&
  \end{array}
\right.
\end{equation*}
By the regularities of $(\hat\rho, \hat u)$ and $(\check\rho, \check u)$, and recalling $\nabla\check u\in L^1(0,T;L^\infty)$, one can easily verify that $\alpha_1,\delta_1,t\beta_1\in L^1((0,T))$.
Therefore, one can apply Lemma \ref{uni-lem} to get $f_1(t)=g_1(t)=G_1(t)=0$, on $(0,T)$, which implies the uniqueness for Case I.

\textbf{Case II: $\rho_\infty>0$. }

By the H\"older and Sobolev inequalities, it follows for $(\rho, u)\in\{(\hat\rho,\hat u),(\check\rho,\check u)\}$ that
\begin{eqnarray*}
  \rho_\infty^4\int|u_t|^2dx&=&\int|\rho_\infty-\rho+\rho|^4|u_t|^2dx\\
&\leq&C\int(|\rho-\rho_\infty|^4+\rho^4)|u_t|^2dx\\
&\leq&C(\|\rho-\rho_\infty\|_6^4\|u_t\|_6^2+\|\rho\|_\infty^3\|\sqrt\rho u_t\|_2^2)\\
&\leq&C(\|\nabla\rho\|_2^4\|\nabla u_t\|_2^2+\|\rho\|_\infty^3\|\sqrt\rho u_t\|_2^2),
\end{eqnarray*}
and, thus,
\begin{equation*}
  \int_0^T t\|u_t\|_2^2dt\leq C\sup_{0\leq t\leq T}(\|\nabla\rho\|_2^4+\|\rho\|_\infty^3)\int_0^T(\|\sqrt t\nabla u_t\|_2^2+\|\sqrt\rho u_t\|_2^2)dt<\infty,
\end{equation*}
that is, $\sqrt tu_t\in L^2(0,T; L^2)$, for $u\in\{\hat u, \check u\}$.

By the H\"older, Sobolev, and Cauchy inequalities, we deduce
\begin{eqnarray*}
  RHS&\leq&\|W\|_2\|\nabla v\|_2+\|\sigma\|_2\|\check u_t\|_3\|v\|_6\\
  &&+\|\sigma\|_2\|\check u\|_6\|\nabla\check u\|_6\|v\|_6+\|\nabla\check u\|_\infty\|\sqrt{\hat\rho}v\|_2^2\\
  &\leq&\|W\|_2\|\nabla v\|_2+\|\nabla\check u\|_\infty\|\sqrt{\hat\rho}v\|_2^2\\
  &&+(\|\check u_t\|_2^{\frac12}\|\nabla\check u_t\|_2^{\frac12}+\|\nabla\check u\|_2\|\nabla^2\check u\|_2)\|\nabla v\|_2\|\sigma\|_2\\
  &\leq&\frac\mu2\|\nabla v\|_2^2+C\|\nabla\check u\|_\infty\|\sqrt{\hat\rho}v\|_2^2+C(1+
\|\check u_t\|_2\|\nabla\check u_t\|_2\\
&&+\|\nabla\check u\|_2^2\|\nabla^2\check u\|_2^2)(\|\sigma\|_2^2+\|W\|_2^2).
\end{eqnarray*}
Plugging this into (\ref{DENGYU}) leads to
\begin{eqnarray}
\nonumber\frac{d}{dt}\|\sqrt{\hat\rho}v\|_2^2+\mu\|\nabla v\|_2^2&\leq  & C\|\nabla\check u\|_\infty\|\sqrt{\hat\rho}v\|_2^2
+C(1+
\|\check u_t\|_2\|\nabla\check u_t\|_2\\
&&+\|\nabla\check u\|_2^2\|\nabla^2\check u\|_2^2)(\|\sigma\|_2^2+\|W\|_2^2).\label{UU2}
\end{eqnarray}

Denote
\begin{eqnarray*}
&&f_2(t)=(\|\sigma\|_2+\|W\|_2)(t),\quad g_2(t)=\|\sqrt{\hat\rho}v\|_2^2(t), \quad G_2(t)=\mu\|\nabla v\|_2^2(t), \\
&&\delta_2(t)=C\|\nabla\check u\|_\infty(t),\quad A_2=C\sup_{0\leq t\leq T}(\|\hat\rho\|_\infty+\|\hat P\|_\infty+\|\nabla\hat\rho\|_2+\|\nabla\hat P\|_3)(t),\\
&&\alpha_2(t)=C\|\nabla\check u\|_\infty(t),\quad \beta_2(t)=C(1+
\|\check u_t\|_2\|\nabla\check u_t\|_2
+\|\nabla\check u\|_2^2\|\nabla^2\check u\|_2^2)(t),
\end{eqnarray*}
then, it follows from (\ref{USIGMA2}), (\ref{UW}), and (\ref{UU2}) that
\begin{equation*}
\left\{
  \begin{array}{ll}
    \frac{d}{dt} f_2(t)\leq \delta_2(t)f_2(t) +A_2 \sqrt{G_2(t)}, & \\
    \frac{d}{dt} g_2(t) +G_2(t) \leq \alpha_2(t)g_2(t) +\beta_2(t)f_2^2(t) , &  \\
    f_2(0)=0.&
  \end{array}
\right.
\end{equation*}
By the regularities of $(\hat\rho, \hat u)$ and $(\check\rho, \check u)$, and recalling $\nabla u\in L^1(0,T; L^\infty)$ and $\sqrt tu_t\in L^2(0,T; L^2)$, for $u\in\{\hat u, \check u\}$, one can easily verify that $\alpha_2,\delta_2,t\beta_2\in L^1((0,T))$.
Therefore, one can apply Lemma \ref{uni-lem} to get $f_2(t)=g_2(t)=G_2(t)=0$, on $(0,T)$, which implies the uniqueness for Case II.

\textbf{Existence:} Set $\rho_{0n}=\rho_0+\frac1n$, $\rho_{n\infty}=\rho_\infty+\frac1n$, and choose $u_{0n}\in D_0^1\cap D^2$, such that $u_{0n}\rightarrow u_0$ in $D_0^1$, as $n\rightarrow\infty$.
Denote
\begin{eqnarray*}
  \psi_0=\|\rho_0\|_\infty+\|\rho_0-\rho_\infty\|_2+\|\nabla\rho_0\|_{L^2\cap L^q}+\|\nabla u_0\|_2,\\
  \psi_{0n}=\|\rho_{0n}\|_\infty+\|\rho_{0n}-\rho_{n\infty}\|_2+\|\nabla\rho_{0n}\|_{L^2\cap L^q}+\|\nabla u_{0n}\|_2.
\end{eqnarray*}
Then, one can easily check that $\psi_{0n}\leq\psi_0+1$, for sufficiently large $n$. By Proposition \ref{Prop9}, there are two positive constants $\mathcal T$ and $C$ depending only on $\mu$, $\lambda$, $a$, $\gamma$, $q$, and $\psi_0$, such that system (\ref{u})--(\ref{rho}), subject to
(\ref{u-bd1}), has a unique solution $(\rho_n, u_n)$, on $\mathbb R^3\times(0,\mathcal T)$, satisfying
\begin{eqnarray}
  \sup_{0\leq t\leq\mathcal T}(\|\rho_n\|_\infty+\|\rho_n-\rho_{n\infty}\|_2+\|\nabla\rho_n\|_2+\|\nabla\rho_n\|_q+\|\partial_t\rho_n\|_2)&\leq& C,\label{EX1} \\
  \sup_{0\leq t\leq\mathcal T} \|\nabla u_n\|_2^2 +\int_0^\mathcal T(\|\nabla^2u_n\|_2^2+\|\sqrt{\rho_n} \partial_tu_n\|_2^2)dt&\leq& C,\label{EX2}  \\
  \sup_{0\leq t\leq\mathcal T}\|\sqrt t\nabla^2u_n\|_2^2
  +\int_0^\mathcal T(\|\sqrt t\partial_t\nabla u_n\|_2^2+\|\sqrt t\nabla^2u_n\|_q^2)dt&\leq& C. \label{EX3}
\end{eqnarray}

Thanks to (\ref{EX1})--(\ref{EX3}), there is a subsequence, still denoted by $(\rho_n, u_n)$, and a pair $(\rho, u)$, satisfying
\begin{eqnarray}
  \label{R1}\rho-\rho_\infty\in L^\infty(0,\mathcal T; H^1\cap W^{1,q}), \quad\rho_t\in L^\infty(0,\mathcal T; L^2), \\
  \label{R2}u\in L^\infty(0,\mathcal T;D_0^1)\cap L^2(0,\mathcal T; D^2),\\
  \label{R3}\sqrt t\nabla^2u\in L^\infty(0,\mathcal T; L^2),\quad
   \sqrt t\nabla u_t\in L^2(0,\mathcal T; L^2),\quad\sqrt t\nabla^2u\in L^2(0,\mathcal T; L^q),
\end{eqnarray}
such that
\begin{eqnarray}
  \rho_n-\rho_{n\infty}\overset{*}{\rightharpoonup}\rho-\rho_\infty,&&\text{ in }L^\infty(0,\mathcal T; H^1\cap W^{1,q}), \label{WCG1}\\
  \partial_t\rho_n\overset{*}{\rightharpoonup}\rho_t, &&\text{ in }L^\infty(0,\mathcal T; L^2), \label{WCG3}\\
  u_n\overset{*}{\rightharpoonup}u,&&\text{ in }L^\infty(0,\mathcal T; D_0^1),\label{WCG4}\\
  u_n \rightharpoonup u,&&\text{ in }L^2(0, \mathcal T; D^2),\label{WCG5}\\
  \partial_tu_n\rightharpoonup u_t,&&\text{ in }L^2(\delta,\mathcal T; D_0^1),\label{WCG6}
\end{eqnarray}
for any $\delta\in(0,\mathcal T)$.
Note that $W^{1,q}\hookrightarrow\hookrightarrow C(\overline{B_k})$, for any positive integer $k$. With the aid of (\ref{WCG1})--(\ref{WCG6}), by the Aubin-Lions lemma, and using the
Cantor's diagonal argument, there is a sequence, still denoted by $(\rho_n, u_n)$, such that
\begin{eqnarray}
  \rho_n\rightarrow\rho,&&\text{ in }C([0,\mathcal T]; C(\overline{B_k})), \label{SCG7}\\
  u_n\rightarrow u,&&\text{ in }L^2(\delta,\mathcal T; H^1(B_k))\cap C([\delta,\mathcal T]; L^2(B_k)),\label{SCG8}
\end{eqnarray}
for any positive integer $k$, and for any $\delta\in(0,\mathcal T)$, where $B_k$ is the ball in $\mathbb R^3$ centered at the origin of radius $k$.
By the aid of (\ref{WCG6}), (\ref{SCG7}) and (\ref{SCG8}), one has
\begin{eqnarray}
  \rho_nu_n\rightarrow\rho u,&&\text{ in }L^2(B_k\times(0,\mathcal T)), \label{NTC1}\\
  \rho_n\partial_tu_n\rightharpoonup\rho u_t, &&\text{ in }L^2(B_k\times(\delta,T)),\label{NTC2}\\
  \rho_n(u_n\cdot\nabla)u_n\rightarrow\rho(u\cdot\nabla)u,&&\text{ in }L^1(B_k\times(\delta,T)),\label{NTC3} \\
  a\rho_n^\gamma\rightarrow a\rho^\gamma,&&\text{ in }C(\overline{B_k}\times[0,\mathcal T]),\label{NTC4}
\end{eqnarray}
for any $\delta\in(0,\mathcal T)$, and for any positive integer $k$.

Due to (\ref{WCG3}), (\ref{WCG5}), and (\ref{NTC1})--(\ref{NTC4}), one can take the limit to the system of $(\rho_n, u_u)$ to show
that $(\rho, u)$ is a strong solution to system (\ref{u})--(\ref{rho}), on $\mathbb R^3\times(0,\mathcal T)$, satisfying the regularities (\ref{R1})--(\ref{R3}). The convergence (\ref{SCG7}) implies that the initial value of $\rho$ is $\rho_0$. The regularity of $\rho-\rho_\infty
\in C([0,\mathcal T]; L^2)$ follows from (\ref{R1}).

The regularity $\sqrt\rho u_t\in L^2(0,\mathcal T; L^2)$ is verified as follows. It follows from (\ref{WCG6}) and (\ref{SCG7}) that
$\sqrt{\rho_n}\partial_tu_n\rightharpoonup\sqrt\rho u_t$ in $L^2(0,\mathcal T; L^2(B_k)),$
for any positive integer $k$. Then, the weakly lower semi-continuity of the norms implies
\begin{equation*}
  \int_0^\mathcal T\|\sqrt\rho u_t\|_{L^2(B_k)}^2dt\leq\liminf_{n\rightarrow\infty}\int_0^\mathcal T\|\sqrt{\rho_n}\partial_tu_n\|_{L^2(B_k)}^2
  dt\leq C,
\end{equation*}
for a positive constant $C$ independent of $k$. Taking $k\rightarrow\infty$ in the above inequality yields
$\sqrt\rho u_t\in L^2(0,\mathcal T; L^2)$.
%To prove
%$\rho u\in C([0,\mathcal T]; L^2)$, noticing that $\rho u\in L^\infty(0,\mathcal T; L^2)$,  it suffices to prove $\partial_t(\rho u)\in L^2(0,\mathcal T; L^2)$.

Finally, we show that $\rho u\in C([0,\mathcal T]; L^2)$ and $\rho u|_{t=0}=\rho_0 u_0$.  By (\ref{rho}) and (\ref{R1})--(\ref{R2}), and noticing that $\|u\|_\infty\leq C\|\nabla u\|_2^{\frac12}\|\nabla^2u\|_2^{\frac12}$, guaranteed by the Gagliardo-Nirenberg and Sobolev embedding inequalities, it follows
\begin{eqnarray}
  &&\int_0^\mathcal T\|\partial_t(\rho u)\|_2^2dt\nonumber\\
  &=& \int_0^\mathcal T
  \|-(u\cdot\nabla\rho+\text{div}\,u\rho)u+\rho u_t\|_2^2dt\nonumber\\
  &\leq&\int_0^\mathcal T\left(\|u\|_\infty^2\|\nabla\rho\|_2+\|u\|_\infty\|\nabla u\|_2\|\rho\|_\infty+\|\rho\|_\infty^{\frac12}\|\sqrt\rho u_t\|_2\right)^2dt\nonumber\\
  &\leq&C\int_0^\mathcal T\left(\|\nabla u\|_2^2\|\nabla^2u\|_2^2 +\|\nabla u\|_2^3\|\nabla^2u\|_2
  + \|\sqrt\rho u_t\|_2^2\right)dt \nonumber\\
  &\leq&C\int_0^\mathcal T(1+\|\nabla^2u\|_2^2+\|\sqrt\rho u_t\|_2^2)dt\leq C. \label{CT1}
\end{eqnarray}
Similarly, it follows from (\ref{EX1})--(\ref{EX2}) that
$\int_0^\mathcal T\|\partial_t(\rho_n u_n)\|_2^2dt\leq C,$ for a positive constant $C$ independent of $n$. Thanks to theses, we
deduce by the H\"older inequality that
\begin{eqnarray*}
  &&\|(\rho u)(\,\cdot\,,t)-\rho_0u_0\|_{L^2(B_R)}\\
  &\leq&\| \rho u - \rho_nu_n \|_{L^2(B_R)}
  +\| \rho_nu_n -\rho_{0n}u_0\|_{L^2(B_R)} +\|\rho_{0n}u_0-\rho_0u_0\|_{L^2(B_R)}\\
  &\leq&\| \rho u - \rho_nu_n\|_{L^2(B_R)}
  +\int_0^t\|\partial_t(\rho_nu_n)\|_{L^2(B_R)}d\tau+\frac{C}{n}\|u_0\|_{L^2(B_R)}\\
  &\leq&\|\rho u-\rho_nu_n\|_{L^2(B_R)}
  +C\sqrt t+\frac{C}{n}\|u_0\|_{L^2(B_R)},
\end{eqnarray*}
for a positive constant $C$ independent of $n$ and $R$. Noticing that
$\rho_nu_n\rightarrow\rho u$ in $C([\delta,\mathcal T]; L^2(B_R)),$ for any $\delta\in(0,\mathcal T)$,
guaranteed by (\ref{SCG7})--(\ref{SCG8}), one can pass the limits $n\rightarrow\infty$ first and then $R\rightarrow\infty$
to the above inequality, and end up with $\|(\rho u)(\,\cdot\,,t)-\rho_0u_0\|_2\leq C\sqrt t.$
This implies $\rho u\in L^\infty(0,\mathcal T; L^2)$ and $\rho u|_{t=0}=\rho_0 u_0$. Thank to these and
recalling (\ref{CT1}), one gets further that $\rho u\in C([0,\mathcal T]; L^2)$. This completes the proof.
\end{proof}

\section*{Acknowledgments}
H. Gong was partially supported by Natural Science Foundation of China (grant No. 11601342, No.11871345, No.61872429) and Natural Science Foundation of SZU (grant No. 2017055). J.Li was partly supported by start-up fund 550-8S0315 of the South China Normal University, the NSFC under 11771156 and 11871005, and the Hong Kong RGC Grant
CUHK-14302917. X. Liu was partial supported by Natural Science Foundation of China(grant No.11631011).

\end{document}